\documentclass[11pt,reqno,makeidx]{amsart}
\usepackage[active]{srcltx}
\usepackage{amssymb,color}
\usepackage{float}
\restylefloat{table}

\usepackage{ragged2e}
\usepackage{array}
\usepackage[active]{srcltx}
\usepackage{epsfig}
\usepackage{hyperref}
\usepackage{amsmath}
\usepackage{amssymb}
\usepackage{hyperref}
\usepackage{enumerate}
\newtheorem{Proposition}{Proposition}
  \newtheorem{Remark}[Proposition]{Remark}
  
  \newtheorem{Lemma}[Proposition]{Lemma}
  
  \newtheorem{Theorem}[Proposition]{Theorem}
 
 \newtheorem{Definition}[Proposition]{Definition}
 \newtheorem{Note}[Proposition]{Note}

\newcommand {\z}{{\noindent}}

\def\Le{\leqslant}
\def\Ge{\geqslant}
 
\def\CC{\mathbb{C}}

 \def\RR{\mathbb{R}}
 \def\NN{\mathbb{N}}
\def\ZZ{\mathbb{Z}}

\def\Re{\mathrm{Re\,}}

\def\P1{P$_{\rm I}$}
\def\3t{_{3t}}

\def\JN{ J}
\def\LN{ L}
\def\newD6p{\mathbb{D}_{21/4}^+}
 \def\({\left(} \def\){\right)} \makeindex

\def\h,{\negthinspace}

\def\z{\noindent}
\def\be{\begin{equation}}
\def\bel{\begin{equation}\label}
\def\ee{\end{equation}}

\def\lam{\lambda}
\def\bd{\begin{Definition}}
\def\ed{\end{Definition}}

\def\bp{\begin{Proposition}}
\def\bpl{\begin{Proposition}\label}
\def\ep{\end{Proposition}}

\def\bl{\begin{Lemma}}
\def\el{\end{Lemma}}
\def\bt{\begin{Theorem}}
\def\et{\end{Theorem}}

\def\bfy{\mathbf{y}}

\def\bfk{\mathbf{k}}

\def\bfs{\mathbf{s}}

\def\lam{\lambda}

\def\oldC0{c_2}
\def\oc1{c_3}
\def\oldc2{c_4}
\def\ogc1{c_5}
\def\olgoc2{c_6}

\def\hsig{{h^\sigma}}

\author{O. Costin, R.D. Costin and M. Huang} \address{Mathematics
  Department\\The Ohio State University\\Columbus, OH 43210}
\address{Mathematics Department\\The Ohio State University\\Columbus, OH
  43210} \address{Mathematics Department\\The University of Chicago, IL 60637
} \title
{Tronqu\'ee solutions of the Painlev\'e equation \P1}
\begin{document}
$ $ \vskip -2cm
\gdef\shorttitle{Tronqu\'ee solutions of \P1}
\maketitle
\today

\definecolor{darkmagenta}{rgb}{.5,0,.5}

\begin{abstract}
  We analyze the one parameter family of tronqu\'ee solutions of the Painlev\'e equation  \P1 in the pole-free sectors together with the region of the first array of poles. We find a convergent expansion for these solutions, containing one free parameter  multiplying exponentially small corrections to the Borel summed power series. We link the position of the poles in the first array to the free parameter, and find the asymptotic expansion of  the  pole positions in this first array (in inverse powers of the independent variable). We show that the tritronqu\'ees are given by the condition that the parameter be zero. We show how this analysis in conjunction with the asymptotic study of the pole sector of the tritronqu\'ee in \cite{inprep} leads to a closed form expression for the Stokes multiplier directly from the Painlev\'e property, not relying on isomonodromic or related type of results.
\end{abstract}

\section{Introduction}

\subsection{ The Painlev\'e equation \P1 and its tronque\'e solutions}\label{sec1}

The Painlev\'e
equation \P1
\begin{equation}\label{p1}
y''= 6y^2 + z
\end{equation}
has a five-fold symmetry: if $y(z)$ solves \eqref{p1}, then so
does  $\rho^2y(\rho z)$ if $\rho^5=1$.

Relatedly, there are five special directions of \eqref{p1} (see also Note\,\ref{SpecDir}) which border the sectors
\begin{equation}
  \label{eq:sectorsz}
  S_k=\left\{z\in\CC\, \Big|\, \frac{2k-1}{5}\pi
<\arg z<\frac{2 k+1}{5}\pi \right\},\ k\in\ZZ_5
\end{equation}
 Generic solutions
have poles accumulating at $\infty$ in all $S_k$.
Any solution has poles in at least one $S_k$ \cite{KitaevKapaev}. For any
{\em two adjacent} sectors $S_k$ there is a one-parameter family of
solutions, called {\em{tronqu\'ee}} solutions, with the behavior
$y=\pm i\frac{\sqrt{z}}{\sqrt{6}}(1+o(1))$ as $z\to\infty$ in these
sectors, and moreover, they are asymptotic to a divergent power series
\begin{equation}\label{tro_ser}
  y(z)\sim \pm i\frac{\sqrt{z}}{\sqrt{6}}\left(1+\sum_{n=1}^\infty\frac{a_n}{z^{5n/2}}\right)\equiv\tilde{y}_0(x)
\end{equation}
whose coefficients $a_n$ are uniquely determined; the free parameter indexing the tronqu\'ee solutions is not visible in their asymptotic behavior \eqref{tro_ser}.

In any set of {\em four}  sectors $S_k$ there is
exactly one solution with behavior \eqref{tro_ser}, see \cite{Kapaev-2004}, \cite{Masoero};
these particular tronqu\'ee solutions, which are maximally regular
solutions (and have poles in only one sector), are called {\em tritronqu\'ees}. The terms "tronqu\'ee", "bitronqu\'ee", and "tritronqu\'ee" together with the corresponding solutions to P1 were first introduced in the pioneering work by Boutroux \cite{Boutroux}. For an overview of the asymptotic behavior of solutions of \P1 equation, and a wealth of references, see \cite{Clarkson}.

The tronqu\'ee solutions we study satisfy
\begin{equation}\label{newsectrinz}
y(z)=i{\frac{\sqrt{z}}{\sqrt{6}}}\,\left(1+o(1)\right)\ \text{as $|z|\to \infty$ with }\arg z\in
  \left[-\frac{3\pi}{5},\frac{\pi}{5}\right]
\end{equation}
and are analytic for large $z$ in this sector.
The other families of tronque\'es are obtained from it by the five-fold symmetry.

 \section{Main results}
\subsection{Overview} In Theorem\,\ref{asir2}  we establish a one-parameter convergent representation of tronque\'e solutions  valid in the pole free sector, using results in \cite{imrn} and \cite{duke}. The natural parameter is the constant multiplying the exponential ``beyond all orders'' of \eqref{tro_ser}. We show that if the constant is zero, then the solution is a tritronqu\'ee solution and describe the Stokes phenomenon for these solutions in Theorem\,\ref{TritCis0}. We find an asymptotic expansion uniformly valid throughout the sector in \eqref{newsectrinz} as well as in a region containing  the first array of poles, (Theorem \ref{FAPIn}) and determine the position of the poles in the first array to $O(x^{-4})$ in Proposition \,\ref{RegP}. These results are obtained using transseries representations of solutions; an overview of this topic is found in the Appendix, \S\ref{IntroTran}-\ref{historic}.

One of the most significant developments in a century of study of
Painlev\'e equations is the isomonodromic approach (and related ones), originating in \cite{Fuchs}, \cite{Garnier}, \cite{Jimbo-Miwa-Ueno}, \cite{Flaschka} and further developed
\cite{Clarkson}, \cite{Deift}, \cite{Deift1}, \cite{joshi},
\cite{Fokas}. A question is whether a complete asymptotic description
and explicit connection formulae are a direct consequence of the
Painlev\'e property, or more structure is needed, e.g., the existence
of an isomonodromic representation; this paper together with
\cite{inprep} are a step toward a positive answer to this question.
In \S\ref{find} we show that the transseries representation together
with a new method to describe solutions in singular regions
\cite{inprep} allows for a closed form calculation of the Stokes
multiplier by direct asymptotic methods using the meromorphicity of
solutions (not relying on isomonodromic-type methods).

\subsection{A convergent one-parameter representation of tronque\'e solutions}

We normalize \eqref{p1} as described in \cite{invent}: the following refinement of the Boutroux substitution
\begin{equation}
  \label{chvarx}
  z={24}^{-1}{30^{4/5}}x^{4/5}e^{-\pi i /5};\ y(z)=i\sqrt{z/6}(1-\tfrac{4}{25}x^{-2}+h(x))
\end{equation}
where the branch of the square root is the usual one, which is positive for
$z>0$, brings \eqref{p1} to the Boutroux-like form
\begin{equation}
  \label{eq:eqp}
h''+\frac{h'}{x}-h-\frac{h^2}{2}-\frac{392}{625}\,\frac{1}{ x^4}=0
\end{equation}
{
This normalization is associated with an interesting maximal regularity property, see Proposition \ref{RegP} below.}

A sector $S_k$  in $z$ corresponds, after the  normalization \eqref{chvarx}, to a quadrant
in $x$,  and the sector $-\pi<{\rm{arg}}\ z\leq \pi$ corresponds to the sector $-\pi<{\rm{arg}}\ x\leq 3\pi/2$. The fact that the solutions of \P1 are meromorphic implies that the solutions of \eqref{eq:eqp} (which have a branch point at $x=0$) return to the same values after analytic continuation around $0$ by an angle of $\tfrac{5\pi}{2}$.

Relatedly,
\eqref{eq:eqp} is invariant under the transformations $h(x)\mapsto h(xe^{\pm i\pi})$ and under the conjugation symmetry $h(x)\mapsto\overline{h(\overline{x})}$.

The tronque\'e solutions \eqref{newsectrinz} correspond to solutions $h$ satisfying
\begin{equation}\label{asyht}
h(x)=o(1)\ \text{as $x\to \infty$ with }\arg x\in \left[-\frac{\pi}{2},\frac{\pi}{2}\right]
\end{equation}


Convergent expansions are obtained using general formal solutions described in Proposition\,\ref{Formh} (see \S\ref{IntroTran}-\ref{Sbeta} for a brief introduction to transseries solutions and their correspondence to actual solutions).

\begin{Proposition}\label{Formh}{\em Formal small solutions of the normalized \P1.}

(i) There is a unique power series solution of \eqref{eq:eqp} which is $o(1)$ as $x\to\infty$, and it has the form
\bel{defhh0}
\tilde{h}_0(x)= \sum_{k=4;\,k\,\text{even}}^\infty c_k x^{-k},\ \ \ \text{with }c_4={-}\frac{392}{625}
\ee

(ii) Transseries solutions of \eqref{eq:eqp} have the form
\bel{transh1}
\tilde{h}(x)=\tilde{h}_{0}(x)+\sum_{k\geq 1}C^k\,e^{-kx}\,\tilde{h}_{k}(x)\ \ \ \ \ \text{for }\arg x\in
(-\tfrac{\pi}{2},\tfrac{\pi}{2})
\ee
where $\tilde{h}_{k}(x)=x^{-k/2}\tilde{t}_{k}(x)$, with $\tilde{t}_{k}(x)$ a nonnegative integer power series in $x^{-1}$,
and
\bel{transh2}
\tilde{h}(x)=\tilde{h}_{0}(x)+\sum_{k\geq 1}C^k\,e^{kx}\,\tilde{\tilde{h}}_{k}(x) \ \ \ \ \ \text{for }\arg x\in
(-\tfrac{\pi}{2},\tfrac{\pi}{2})\pm i\pi
\ee
 where $\tilde{\tilde{h}}_{k}(x)=x^{-k/2}\,e^{\mp ik\pi/2}\,\tilde{t}_{k}(-x)$.

 \end{Proposition}

The proof of Proposition\,\ref{Formh} is found in \S\ref{PfFormh}.

We note that $\tilde{h}_n(x)$ are linked, via \eqref{chvarx}, to the $n$-instanton of \P1 \cite{Its_G_K}.

{\bf{Notations.}} In the following the
Laplace transform of  $Y(p)$ in the direction
$e^{i\phi}$ is defined as
\bel{DefLphi}
\mathcal{L}_\phi Y\,(x)=\int_{e^{i\phi}\RR_+} {\rm{e}}^{-px}\, Y(p)\, dp
\ee
where, by convention, $\phi=-\arg x$. The convolution is defined as $(f*g)(p)=\int_0^p f(s)g(p-s)ds$. The Borel transform of a series is, as usual,
\bel{DefB}
\mathcal{B}\left(x^{-\alpha}\sum_{n=0}^\infty c_nx^{-n}\right)=p^{\alpha}\sum_{n=0}^\infty \,c_{n}\,\frac{p^{n-1}}{\Gamma(n+\alpha)}\ \ \ \ \text{for }\alpha>0
\ee
(For an introduction to the Borel transform and Borel summation see \S\ref{LBsum}.)

 Next theorem describes the tronqu\'e solutions \eqref{newsectrinz} in coordinates \eqref{chvarx}.

\begin{Theorem}[Tronqu\'ees as Borel summed transseries]\label{asir2}\ \

I. Assume $h$ solves \eqref{eq:eqp} and satisfies
\begin{equation}\label{assph}
h(x)=o(1)\ \text{as }
  x\to\infty\text{ with } \arg(x)\in (-\tfrac{\pi}2,\tfrac{\pi}{2})
  \end{equation}
Then

 (i)  We have $h\sim \tilde{h}_0$
as $x\to +\infty$ and the asymptotic expansion is differentiable.

 Also, there exist constants $C_\pm$ so that
 \begin{equation}
 \begin{array}{l}
 h(x)\sim C_+x^{-1/2}e^{-x}\ \text{as}\ x\to +i\infty, \\ \\
h(x)\sim C_-x^{-1/2}e^{-x}\ \text{as}\ x\to -i\infty
\end{array}
\end{equation}

 (ii) Let
 \begin{equation}\label{defHk}
 H_k=\mathcal{B}\tilde{h}_k
 \end{equation}
 be the Borel transforms of the series in the transseries solution \eqref{transh1}.

Then $h(x)$ has the Borel summed transseries representations:
 \begin{equation}
    \label{eq:transh}
    h(x)=\left\{
    \begin{array}{l}
   \mathcal{L}_\phi H_0\,(x)+ \sum_{k=1}^{\infty} C_+^k e^{-kx}\,\mathcal{L}_\phi H_k\,(x)\ \ \ \text{for }
   - \phi= \arg x\in(0,\tfrac{\pi}{2}]\\ \\
    \mathcal{L}_\phi H_0\,(x)+\sum_{k=1}^{\infty} C_-^k e^{-kx}\,\mathcal{L}_\phi H_k\,(x)\ \ \ \text{for }
    -\phi= \arg\, x\in[-\tfrac{\pi}{2},0)
    \end{array}\right.
    \end{equation}
where $\mathcal{L}_\phi H_k$ are analytic for large $x$ with  $\arg x\in (-\pi/2, 3\pi/2)$ and the series converge for $|x|$ large enough with $0<|\arg x|\Le\tfrac{\pi}{2}$; $\mathcal{L}_\phi H_k=O(x^{-k/2})$ in these regions.

II. Similar Borel summed transseries representations hold for small solutions in the sectors $\arg(x)\in (-\tfrac{\pi}2,\tfrac{\pi}{2})\pm i\pi$, with $e^{-kx}$ replaced by $e^{kx}$, all $H_k(\cdot)$ with $k\Ge 1$ replaced by $-H_k(\cdot e^{\pm i\pi})$ and with the constants $C$ changing only at the Stokes lines $\arg x=\pm \pi$.

\end{Theorem}

The proof of Theorem\,\ref{asir2}  is found in \S\ref{PfTasir2}.

\subsection{Description of the Stokes phenomenon}\label{S2.2}

The proposition below links the Stokes constant $C_+-C_-$ in \eqref{eq:transh} to the leading behavior of $H_0=\mathcal{B}\tilde{h}_0$ at $p=1$.

\begin{Proposition}\label{H0_behave}

(i) Near $p=1$ $H_0(p)$ has the form ${H}_0(p)=S\,(1-p)^{-1/2}{A}(p)+{B}(p)$
with ${A}(p),\, {B}(p)$ analytic at $p=1$, ${A}(p)=1+O(p-1)$ and $S$ is a constant.

(ii) Denote
\begin{equation}\label{defhpm}
\displaystyle{\begin{array}{l} h^+(x)=\mathcal{L}_{\phi}H_0\,(x)\ \text{for }-\phi=\arg(x)\in(0,\tfrac{\pi}{2})\\ \\
                                                         h^-(x)=\mathcal{L}_{\phi}H_0\,(x)\ \text{for }-\phi=\arg(x)\in(-\tfrac{\pi}{2},0)
                                                         \end{array} }
                                                         \end{equation}

Then $h^+(x)$ can be analytically continued for large $x$ with $\arg x\in[-\tfrac{\pi}{2},\pi]$, $h^-(x)$ can be analytically continued for large $x$ with $\arg x\in[-\pi,\tfrac{\pi}{2}]$ and
\begin{equation}\label{jumph01}
h^+(x)-h^-(x)= -\mu e^{-x}x^{-1/2}(1+o(1))\ {\rm{as\ }}x\to+\infty\ \ \text{where }\ \ \mu=2iS\sqrt{\pi}
\end{equation}
with $S$ defined in (i).

(iii) The constants $C_\pm$ in \eqref{eq:transh} satisfy
$C_+-C_-=-\mu$
with $\mu$ as in \eqref{jumph01}.

  \end{Proposition}

 The proof of {Proposition}\,\ref{H0_behave} is found in \S\ref{PfP3}.
 \begin{Note}{\rm The {\em Stokes constant} $\mu$ was calculated in closed form, $  \mu=i\sqrt{\frac{6}{5 \pi }}$
    first by Kapaev using the method of isomonodromic deformations  \cite{Kapaev} (some corrections were made in \cite{KitaevKapaev}).
   The {\em existence} a constant $\mu$  (independent of $C_\pm$) such that $C_--C_+=\mu$ is known for  a wide class of
   differential equations, see  formula following (1.15) in \cite{imrn},
   and also  \cite{duke}, \eqref{eq:sjump}.  An {\em{explicit expression}} for $\mu$ is expected only in special cases such as
integrable equations. The explicit value also follows, without isomonodromic-type methods from the asymptotic analysis in this paper together with \cite{inprep}.}

 \end{Note}

\subsection{Arrays of poles near regular sectors of tronqu\'ee solutions}\label{firstpoles}
Theorem\,\ref{FAPIn} describes  the asymptotic behavior of tronqu\'ee solutions for large $x$ in the pole free sectors together with the region of the first array of poles: \eqref{eq:eq51} and \eqref{eq:devh1} give a uniform expansion of the tronqu\'ees in the sector of analyticity up to and including the first array of poles near $i\RR^+$. Based on this, Proposition\,\ref{RegP} gives the position of these poles, and the way it depends on the value of the constant beyond all orders $C_+$.

\begin{Theorem}\label{FAPIn}\ { (i) \cite{invent}
Let $h(x)$ be a solution as in Theorem\,\ref{asir2}\,I with {$C_+\ne 0$.} Denote
\begin{equation}
  \label{eq:eqdefxi}
  \xi=C_+x^{-1/2}e^{-x}
\end{equation}
Then the leading behavior $h(x)$ for large $|x|$ with ${\rm{arg}}\, x$ close to $\tfrac{\pi}{2}$  is
 \begin{equation}
   \label{eq:eq51}
   h\sim F_0(\xi)+\frac{F_{1}(\xi)}{x}+\frac{F_{2}(\xi)}{x^2}+\cdots\ \ (|x|\to \infty,\ x\in\mathcal{D}_{x})
    \end{equation}
where
\begin{equation}\label{regZx}
{ \mathcal{D}_{x}=\left\{x\in\CC\,|\, |x|>R,\,\arg x\in (-\tfrac{\pi}{2}+\delta,\tfrac{\pi}{2}+\delta),\,|\xi(x)-12|>\epsilon,\, |\xi(x)|<\epsilon^{-1}\right\} }
 \end{equation}
for any $\delta,\epsilon>0$ small enough and $R=R(\epsilon,\delta)$ large enough, and
where
\begin{equation}\label{eq:f11}
 F_0(\xi)= \tfrac{144\, \xi}{(\xi-12)^2},\, F_{1}(\xi)=\tfrac{
\frac{1}{60}\xi^4-3\xi^3-210\xi^2
-216\xi}{(\xi-12)^3},\ldots,\, F_{n}(\xi)=\tfrac{P_n(\xi)}{(\xi-12)^{n+2}}
\end{equation}
with $P_n$ polynomials of degree $2n+2$ for $n\Ge 1$.

(ii)   An  expansion similar to \eqref{eq:eq51} holds for $g=h(1+\frac13h)^{-1}$ in the region $\mathcal{D}'_x$ given by \eqref{regZx} with $12$ replaced by $-12$:
   \begin{equation}
     \label{eq:devh1}
     g\sim G_0(\xi)+\frac{1}{x}G_1(\xi)+\frac{1}{x^2}G_2(\xi)+\cdots
   \end{equation}
   with
   $$G_0(\xi)=\frac{144\xi}{(\xi+12)^2},\ \ G_1(\xi)=-\frac{1}{60}\frac{\xi (\xi-12)(\xi^3-180\xi^2-12600\xi-12960)}{(\xi+12)^4}$$
 At any point in the region $|\xi(x)|<\epsilon^{-1}$ one, or both
of the expansions \eqref{eq:eq51}, \eqref{eq:devh1} holds.}
(iii) The symmetry $h\to\overline{h({\overline{x}}})$ implies that a similar representation exists in the fourth quadrant with $C_-$ instead of $C_+$.
\end{Theorem}

The proof is found in \S\ref{PfRem8}.
The counterpart of expansion \eqref{eq:eq51} in the original variables of P$_I$ is given in \cite{invent}.

\begin{Proposition}[Position of poles]\label{RegP}\ \newline
(i) Any solution as in Theorem\,\ref{asir2} with $C_+\ne 0$ in \eqref{eq:transh} is analytic for large $x$ in the region $|\xi(x)|<12$ and has a first array of poles near the curve  $\xi(x)=12$, located at
\begin{equation}
  \label{eq:pospoles}
 x_n=2n\pi i+L-\frac{\frac{109}{120}+\frac12 L}{2 n\pi i}+\frac{\frac{4699}{2400}+\frac{139}{120}L+\frac14L^4}{(2 n\pi i)^2}\\ - \frac{\frac{41402111}{6480000}-\frac{899}{200}L-\frac{77}{60}L^2-\frac16 L^3}{(2 n\pi i)^3}+O(n^{-4})
\end{equation}
for $n\to\infty$, where $  L=\displaystyle \ln\(\frac{C^+}{12(2 n\pi i)^{1/2}}\)$.

(ii) The following maximal regularity property holds. If $\mathcal{A}$ is an analytic function tangent to the identity, $\mathcal{A}(0)=0,\mathcal{A}'(0)=1$, then $\mathcal{A}(h(x))$ is singular for some large $x$ in the region $|\xi(x)|<R$ if $R>12$.

\end{Proposition}

The proof of Proposition\,\ref{RegP} is found in \S\ref{PfPr10}.

{\em Note:} rotating $x$ further into the
second quadrant, $h$ develops
successive arrays of poles   separated by distances
$O(\ln x)$ of each other as long as $\arg(x)=\tfrac{\pi}{2}+o(1)$ \cite{invent}.

\subsection{Tritronqu\'e solutions}

The following theorem characterizes tritronqu\'e solutions as the tronqu\'ees with a zero constant beyond all orders and establishes a representation as convergent series in all four pole free sectors.
\begin{Theorem}\label{TritCis0}
(i) The tritronque\'e  solution $y_t$ of \eqref{p1} with
\begin{equation}\label{sectorinz}
y_t(z)=i\sqrt{\frac{z}{6}}\,\left(1+o(1)\right)\ \text{as $|z|\to \infty$ with }\arg z\in
  \left(-\frac{3\pi}{5},\pi\right)
\end{equation}
is the unique solution analytic for large $x$ in some sector in the first quadrant, having $C_+=0$ in its representation \eqref{eq:transh}.

(ii) Define $\hsig=\mathcal{L}_\phi H_0$ for $-\phi=\arg x\in(\pi,\tfrac{3\pi}{2})$ and let $h^+(x), h^-(x)$ be as in \eqref{defhpm}.

Then  $h^+$ and $\hsig$ can be analytically continued to $\arg x=\pi$ and
\begin{equation}\label{jumphpi}
h^+(x)-\hsig(x)= \mu e^{-|x|}|x|^{-1/2}(1+o(1))\ {\rm{for\ }}x\to-\infty
\end{equation}
with $\mu$ as in \eqref{jumph01}.

(iii) We have
\begin{equation}\label{ht_trans}
    h_t(x)=\left\{
    \begin{array}{l}h^+(x)
   \ \ \ \text{for }
    \arg x=- \phi\in(0,\pi)\\ \\
    h^-(x)+\sum_{k=1}^{\infty} (-\mu)^k e^{-kx}\,\mathcal{L}_\phi H_k\,(x)\ \ \ \text{for }
     \arg x=- \phi\in[-\tfrac{\pi}{2},0)\\ \\
    \hsig(x)+\sum_{k=1}^{\infty} (i\mu)^k e^{kx}\,\mathcal{L}_\phi H_k\,(x)\ \ \ \text{for }
     \arg x=- \phi\in(\pi,\tfrac{3\pi}{2})
    \end{array}\right.
    \end{equation}

\end{Theorem}
The proof of Theorem\,\ref{TritCis0}  and the choice of branches are given  in \S\ref{PfTritCis0}.
\begin{Note} (i) The uniqueness of $h_t$ and the symmetry of the equation imply $h_t(x)=\overline{h_t(-\overline{x})}$, and the last expansion in \eqref{ht_trans} follows from the middle one, symmetry and the choice of branches.

(ii) Also by symmetry, we see that $C_-=0$ corresponds to tritronqu\'ee solutions which are pole free in
the sector $\arg z\in
  (-\frac{7\pi}{5},\frac{\pi}{5})$.
\end{Note}

\begin{Remark}
  {\rm  $y_t$ is now known to be pole free in the whole sector, up to the origin ({\em{Dubrovin's conjecture}} holds); more precisely, $y_t$ is analytic in a ball near the origin and in the closed sector $\{z:||z|>0, \arg z \in [-3\pi/5,\pi]\}$, \cite{DubrovinC}}.
\end{Remark}

\subsection{Calculating the constant beyond all orders from the values of the tronqu\'ee}
The tronqu\'ee solutions are distinguished by the constant $C_\pm$ beyond all orders. We recall the following result which allows for the calculation of $C_{\pm}$ using the actual solution.

 \begin{Theorem} \cite{OC-MDK} The constant $C_\pm$ in \eqref{eq:transh}  satisfies
   \begin{equation}
     \label{eq:eqCpm}
     C_\pm=\lim_{\begin{subarray}{c}x\rightarrow\infty\\\arg(x)=-\phi\end{subarray}}\, e^x\, x^{1/2}\left(h(x)-\sum_{{k\le |x|}}\frac{c_{k}}{x^k}\right)
   \end{equation}
(see \eqref{defhh0}) for $x$ in the corresponding quadrant.

In the direction $\arg x=0$ the limit  \eqref{eq:eqCpm} is $\tfrac{1}{2}C_++\tfrac{1}{2}C_-$.
\end{Theorem}

\section{Application: finding the Stokes multiplier}\label{find}

Let $y_t$ be the tritronqu\'ee of Theorem\,\ref{TritCis0}. The continuation of $y_t(z)$ for $\arg z$ from $-\tfrac{3\pi}{5}$ up to $-\pi$, through the pole sector, is given with all technical details in \cite{inprep}. We describe below the method used, the intuition and formal calculations behind  \cite{inprep}.

After normalization, $y_t(z)$ corresponds to $h_t(x)$ analytic for
large $x$ in the sector $\arg x\in(-\tfrac{\pi}{2},\tfrac{3\pi}{2})$ and having poles in the sector
 \begin{equation}\label{def-Sigma}
  \Sigma=\{x\, |\, -\pi<{\rm{arg}}\, x<-\pi/2\}
  \end{equation}
The first array of poles near the edges of this sector is described in  \S\ref{firstpoles}.

 We show that in the pole region \eqref{def-Sigma} $h_t$ can be described by constants of motion which are
valid starting with $\arg x$ close to $-\tfrac{\pi}{2}$ (the first array of poles) up to $\arg x$ close to $-\pi$, close to another "first" array of poles, where $h_t$ can be again {\em matched to a transseries}; the asymptotic expansions of the
constants of motion that we obtain depend explicitly on the constant beyond all orders $\mu$. The transseries
representation of $h_t$ also depends on  $\mu$ in a way visible in the leading order asymptotics when  $\arg x=-\pi$ or $3\pi/2$, cf. Theorem \ref{asir2} (we note that both arguments of $x$ correspond to $z\in\RR^-$).

\subsection{The connection problem} The solution $y_t$ is meromorphic; this
was known since Painlev\'e, and proving meromorphicity does not require a
Riemann-Hilbert reformulation, see e.g. \cite{OCRDCP1,Hinkkanen,Gromak} for direct
proofs and references.
  Starting with  a large $z\in\RR^+$ we analytically continue $y_t$
(i) anticlockwise on an arccircle until $\arg z=\pi$   and (ii) clockwise on
an arccircle until $\arg z=-\pi$. The continuation (ii) traverses the pole
sector, $\arg z\in (-\pi,-3\pi/5)$.
Because of the above-mentioned
meromorphicity, we must have
\begin{equation}
  \label{eq:mer1}
  y_t(|z|e^{i\pi})=y_t(|z|e^{-i\pi})
\end{equation}
In variables \eqref{chvarx} this corresponds to the following. We start with large
$x$ with  $\arg x=\tfrac{\pi}{4}$ and (i') analytically continue $h_t(x)$ anticlockwise, until $\arg
x=\tfrac{3\pi}{2}$, and also (ii')  clockwise, until $\arg x=-\pi$. The
single-valuedness equation \eqref{eq:mer1} and eq. \eqref{chvarx} imply
 \begin{equation}\label{hpm}
h_t(|x|e^{3\pi i/2}) = -h_t(|x|e^{-\pi i})-2+\frac{8}{25|x|^2}
\end{equation}
Eq. \eqref{hpm} implies a nontrivial
equation for $\mu$, \eqref{stok2} below, explained in \S\ref{polreg},  and which determines $\mu$ uniquely. The fact that $\mu$ is uniquely
determined relates to the fact that there is only one solution, the
tritronqu\'ee, with
algebraic behavior in the region \eqref{sectorinz}, cf. \cite{invent}, Proposition
15. Relatedly, the last expansion in \eqref{ht_trans} is not used in the calculation.

\subsection{Asymptotics of the tronque\'es in the pole region $\arg x\in (-\pi,-\pi/2)$: heuristics}\label{polreg}
Our calculations rely on the construction of two functionally independent adiabatic invariants. In a nutshell, an adiabatic invariant is a conserved quantity given as an asymptotic expansion with controlled errors (the expansion here is in powers of $1/x$.) Adiabatic invariants are useful when dealing with small perturbations of integrable systems (usually in the stronger sense of explicit integrability). In our problem, for large $x$, \P1 is treated as a small perturbation of the elliptic equation $f''-f-f^2/2=0$. Relying on adiabatic invariants is reminiscent of KAM techniques: indeed, conserved quantities are the complex-analogs of  action-angle variables.  To our knowledge the adiabatic invariants (48) and (50), refined in \cite{inprep} are new. Their construction is facilitated by the Boutroux ``cycle'' technique \cite{Boutroux}.

 Note that for
large $x$ \eqref{eq:eqp} is close to the autonomous Hamiltonian system
\begin{equation}
  \label{eq:eqell}
  h''-h-h^2/2=0\ \ \text{with Hamiltonian $s/2$}
\end{equation}
where
\begin{equation}
  \label{def_s}
  s={h'}^2-h^2-h^3/3
\end{equation}
The solutions of \eqref{eq:eqell} are elliptic functions, doubly
periodic in $\CC$. For \eqref{eq:eqp} we expect solutions to be
asymptotically periodic, and $s$ to be a slow varying quantity; this
is certainly the case in the region where \eqref{eq:eq51} holds. It is
then natural to take $h=:u$ as an independent {\em angle-like}
variable and treat $s$ and $x$ as dependent variables.  With $w=u'$ we
first rewrite equation (\ref{eq:eqp}) as a system
\begin{align}
\label{eqqq1}u' &=w\\
\label{eqqq2}w' &=u+\frac{u^2}{2}-\frac{w}{x}+\frac{392}{625}\,\frac{1}{ x^4}
\end{align}
and then, with
\begin{equation}\label{def_R}
R(u,s)=\sqrt{u^3/3+u^2+s}
\end{equation}
we transform \eqref{eqqq1}, \eqref{eqqq2}
into a system for $s(u)$ and $x(u)$:
\begin{align}\label{syst2}
&\frac{ds}{du}=-\frac{2w}{x}+\frac{784}{625}\,\frac{1}{x^4}=-\frac{2R(u,s)}{x}+\frac{784}{625}\,\frac{1}{x^4}\\
\label{syst3}
&\frac{dx}{du}=\frac{1}{w}=\frac{1}{R(u,s)}
  \end{align}

 Let $h$ be a tronqu\'ee solution of \eqref{eq:eqp} having poles for large $x$ in the sector $\Sigma$ in \eqref{def-Sigma}.
 It turns out that there are closed curves $\mathcal{C}$, similar to the classical
  {\em cycles} \cite{Kitaev}, such that $R(u,s(u))$
  does not vanish on $\mathcal{C}$ and $x(u)$
  traverses $\Sigma$ from edge to edge as $u$ travels along  $\mathcal{C}$ a number $N$
  times.

   Written in integral form, \eqref{syst2} and \eqref{syst3} become
\begin{align}
  \label{eq:eqds2n}
 & s(u)=s_n-2\int_{u_n}^u \(\frac{R(v,s(v))}{x(v)}-\frac{392}{625}\frac{1}{x(v)^4}\)dv \\
 & x(u)=x_n+\int_{u_n}^u \frac{1}{R(v,s(v))}dv \label{eq:eqdx2n}
\end{align}
where the integrals are along the path $\mathcal{C}$ and we write $u_n$ to denote that $u$ has traveled $n$ times along $\mathcal{C}$, and $s_n=s(u_n)$ and
  $x_n=x(u_n)$.

\subsubsection{Initial conditions and iteration}\label{IC1} We take  initial values  $x_0,\, s_0$ so that $|x_0|$ is large and  close  to the
  first array of poles (see \S\ref{firstpoles}) i.e. $\arg(x_0)=-\pi/2(1+o(1))$ for large $|x_0|$, and $s_0$ sufficiently small; $x_0,s_0$ are arbitrary otherwise.

  \begin{Note}\label{IC2}
    The parameters $x_0, s_0$ are free constants in an open set in $\CC^2.$ By Theorem \ref{FAPIn} (iii),  \eqref{eq:eq51} and \eqref{eq:eqdefxi} we see that for the tronqu\'ees $x_0, s_0$ depend on the constant beyond all orders $C_-$.
  \end{Note}

   Starting with  $u_0\in\mathcal{C}$  the  following hold\footnote{Later on we will choose $u_0=-4\in\mathcal{C}$.}. For some $N=N_{m}(x_0)$ of the order $|x_0|$, $x_N$ is close to
  the last array of poles (i.e. $\arg(x_N)=-\pi(1+o(1))$), and
  $|x_n|$  is of the order $|x_0|$ for all $n\leq N$.  Also, two roots
  of $R(u_n,s_n),n=0,1..,N$ are in the interior of $\mathcal{C}$ and a third one is in
  its exterior.

To establish these facts, an important ingredient in \cite{inprep}  is the study of the Poincar\'e
map for \eqref{eq:eqds2n}, \eqref{eq:eqdx2n}, namely the study of
$(s_{n+1},x_{n+1})$ as a function of $(x_n,s_n)$. The Poincar\'e map is used to
eliminate the fast evolution. The asymptotic expansions of $s(u)$ and $x(u)$
when $u$ is between $u_n$ and $u_{n+1}$ are straightforward local expansions
of \eqref{eq:eqds2n} and \eqref{eq:eqdx2n}.  We denote
\begin{equation}
  \label{eq:eqJL}
    \JN (s)=\oint _\mathcal{C} R(v,s)\,dv;\ \ \LN(s)=\oint_\mathcal{C} \,\frac{dv}{R(v,s)}
\end{equation}
It is easily checked that
\begin{equation}
  \label{eq:difeq00}
 \JN\,''+\frac{1}{4}\rho(s)\JN=0;\ \ \text{where}\ \  \rho(s)=\frac {5}{3s \left( 3\,s+4 \right) }
\end{equation}
and, since $\JN\,'=\LN/2$ we get
\begin{equation}
  \label{eq:difeq01}
  \LN\, ''-\frac{\rho'(s)}{\rho(s)}\LN\, '+\frac{1}{4}\rho(s)\LN=0
\end{equation}

The points $s=0$ and $s=-4/3$ are regular singular points of
\eqref{eq:difeq00} (and of \eqref{eq:difeq01}) and correspond to the
values of $s$ for which the polynomial $u^3/3+u^2+s$ has repeated
roots.
 Simple asymptotic analysis
of \eqref{eq:eqds2n} and \eqref{eq:eqdx2n} shows that the Poincar\'e map satisfies
 \begin{equation}\label{sna0}
s_{n+1}=s_n-\frac{2J_n}{x_n}\left(1+{o}(1)\right)\ \ \ \ \ \ \ {\rm{with\ }}J_n=\JN(s_n)
\end{equation}
\begin{equation}\label{xna0}
x_{n+1}=x_n+L_n\left(1+{o}(1)\right)\ \ \ \ \ \ \ {\rm{with\ }}L_n=\LN(s_n)
\end{equation}
Here, and in the following {\em heuristic} outline, $o(1)$ stands for terms
which are small for large $x_n$ and large $n$. The rigorous justification of
these estimates is the subject of \cite{inprep}.
\subsubsection{Solving \eqref{sna0} and \eqref{xna0}; asymptotically conserved quantities}
We see from \eqref{sna0} that $s_{n+1}-s_n\ll s_n$ and $x_{n+1}-x_n\ll x_n$.
It is natural to take a ``continuum limit'' and
approximate $s_{n+1}-s_n\text{ by } ds/dn$ and $x_{n+1}-x_n\text { by }dx/dn$.
We get
\begin{equation}\label{eqapprx1}
\frac{ds}{dx}=\frac{ds/dn}{dx/dn}=\frac{-2\JN(s)}{x\LN(s)}\left(1+o(1)\right)=-\frac{\JN(s)}{x\JN\, '(s)}\left(1+o(1)\right)
\end{equation}
which implies, by separation of variables and integration,
\begin{equation}
  \label{eq:eqapprx}
  \mathcal{Q}(x,s):= x\JN(s)= x_0\JN(s_0)\, \left(1+o(1)\right)=\mathcal{Q}(x_0,s_0)\, \left(1+o(1)\right)
\end{equation}
That is, $\mathcal{Q}$ is asymptotically constant. A second
(nonautonomous) one is obtained using \eqref{sna0} and \eqref{eq:eqapprx} as
follows. We write
\begin{equation}\label{eq:2ndc1}
  \frac{1}{\JN\, ^2(s)}\,  \frac{ds}{dn}=-\frac{2}{x_0\JN(s_0)}\left(1+o(1)\right)
\end{equation}
Let  $\hat{J}$ be an independent solution of \eqref{eq:difeq00}  with $\hat{J}(0)=0$ and denote
\begin{equation}\label{defK}
\mathcal{K}(s):=\kappa_0\int_{0}^s\frac{ds}{\JN(s)^2}
=\frac{\hat{J}(s)}{\JN(s)}
\end{equation}
where  $\kappa_0$ is
the Wronskian of $\JN$ and $\hat{J}$. Integrating both sides of \eqref{eq:2ndc1}
from $0$ to $n$ we get
\begin{equation}\label{eq:2ndc}
\mathcal{K}(s)-\mathcal{K}(s_0) =-\frac{2n}{\kappa_0\, x_0\JN(s_0)}\,\left(1+o(1)\right)\Rightarrow  \mathcal{K}(s)+\frac{2n}{\kappa_0\, x_0\JN(s_0)}=\mathcal{K}(s_0)+o(1)
\end{equation}
(for $n=O(x_0)$).
 $\mathcal{K}$ is  a Schwarzian triangle function.
We thus obtain two functionally independent  asymptotically conserved quantities \eqref{eq:eqapprx},\,\eqref{eq:2ndc}
from which we can retrieve the asymptotics of the tronqu\'ee solutions in the pole sector.
 The rigorous proof is the subject of \cite{inprep}, where three asymptotic orders of the expansion are obtained. Near the antistokes lines the expansion takes a slightly different form.

 \begin{Note}{\rm
     From Note \ref{IC2} it is seen that in the pole region, the classical asymptotic expansion contains two free constants (which for the tronqu\'ees depend on $C_-$).  The fact that the constants are classically visible makes possible to calculate their {\em change} as the pole region is traversed, and thus to {\em calculate explicitly $\mu$} from the requirement that $h_t$ has a transseries representation at both edges of one pole sector. This condition leads to the equation (\cite{inprep}):
} \end{Note}
\begin{multline}\label{stok2}
-\frac{4 \sqrt{3}-24 i}{5 \pi}=\frac{24 i N}{5 \pi}+\frac{1}{5 \pi ^2}\bigg[12 \ln \left(\frac{(6+6 i) \left(\sqrt{3}+i\right)}{\mu}\right)+i \pi -4 \sqrt{3} \pi
-6 \ln (240 \pi )\bigg]
\end{multline}
A straightforward  calculation shows that \eqref{stok2} implies
\begin{equation}
  \label{eq:mu-value}
  \mu=\sqrt{\frac{6}{5 \pi }}i
\end{equation}

\section{Proofs}

\subsection{Normal form for \eqref{eq:eqp}}

We first transform \eqref{eq:eqp} to a normal form, to which we apply the general results of \cite{imrn}, \cite{duke}, then use this information to obtain results about the tronqu\'ee solutions of  \eqref{eq:eqp}.

Simple algebra brings \eqref{eq:eqp} to the normal form

\begin{equation}\label{sysP1}
\mathbf{y}'+\left({\Lambda}+\frac{1}{x}{B}\right)\mathbf{y}=\mathbf{g}(x^{-1},\mathbf{y})
\end{equation}
with ${\Lambda}=\text{diag}(\lambda_1,\lambda_2),{B}=\text{diag}(\beta_1,\beta_2)$ where
\begin{equation}\label{noP1}
  \lambda_{1}=1,\ \lambda_{2}=-1,\ \beta_{1}=\beta_{2}=\tfrac12
  \end{equation}
and $\mathbf{g}=(g_1,g_2)^T$ is analytic in $(x^{-1},\bfy)$ in a neighborhood of zero and
$\mathbf{g}=O(x^{-4})+O(\bfy x^{-2})+O(|\mathbf{y}|^2)$ as $x^{-1},\bfy\to 0$; see \S\ref{rewrP1} for details.

\subsection{Transseries solutions}\label{PfFormh}

For an introduction to transseries see \S\ref{IntroTran}.

\begin{Proposition}{\em General formal solution of the system}

The transseries solutions of the system \eqref{sysP1}-\eqref{noP1} are the following.

I. For $\arg x\in\left(-\frac{\pi}{2},\frac{\pi}{2}\right)$ the transseries solutions are
\bel{tranyt}
\tilde{\bfy}(x)=\tilde{\bfy}_0(x)+\sum_{k=1}^\infty C^k\,e^{-kx}\,\tilde{\bfy}_k(x)\ \ \ \text{with }\tilde{\bfy}_k(x)=x^{-k/2}\, \tilde{\bfs}_k(x)
\ee
where $C$ is an arbitrary constant, $\tilde{\bfs}_k(x)$ is an entire power series in $x^{-1}$ and
$$\tilde{\bfy}_0= \sum_{k\Ge 4}\left(\begin{array}{c}1\\(-1)^k\end{array}\right)c_k x^{-k},\ \ \ \ \tilde{\bfs}_{1}=\left(1+\frac{1}{8x}\right)\left(\begin{array}{c}1\\0\end{array}\right)+O(x^{-2})$$

II. For $\arg x\in\left(\frac{\pi}{2},\frac{3\pi}{2}\right)$  the transseries solutions are
\bel{tranytt}
\tilde{\bfy}(x)=\tilde{\tilde{\bfy}}_0(x)+\sum_{k=1}^\infty C^n\,e^{kx}\,\tilde{\tilde{\bfy}}_k(x)\ \ \ \text{with }\tilde{\tilde{\bfy}}_k(x)= x^{-k/2}\,\tilde{\tilde{\bfs}}_k(x)
 \ee
where $C$ is an arbitrary constant, and $\tilde{\tilde{\bfs}}_k(x)={\tilde{\bfs}}_k(-x)$, $\tilde{\tilde{\bfy}}_0(x)={\tilde{\bfy}}_0(-x)$.
\end{Proposition}

{\em Proof.} This is an application of Theorem 2 in \cite{imrn}.

\subsubsection{Proof of Proposition\,\ref{Formh}}
(i) is obtained by a straightforward calculation. Furthermore, this follows from the fact that \eqref{eq:eqp} is a re-writing of \P1, whose asymptotic expansions are known \cite{Clarkson}.

The general formal solution of systems with a rank one irregular singularity has the form \eqref{gtrans} below; the system \eqref{sysP1}-\eqref{noP1} has $d=2$ and the transseries must have $C_1$ or $C_2$ equal to zero. Since by \eqref{eq:trsf3} below we have
\bel{htoy}
h=\frac{1}{2}\left(1-\frac{1}{4x}\right)y_1+\frac{1}{2}\left(1+\frac{1}{4x}\right)y_2
\ee
part (ii) follows. $\Box$

\subsection{Proof of Theorem\,\ref{asir2}}\label{PfTasir2}

After establishing Lemma\,\ref{L42}, we apply \cite{imrn} to the system \eqref{sysP1}-\eqref{noP1}; we then establish properties of the Borel transform of the series $\tilde{h}_k$ in Lemma\,\ref{Lemma2}, then complete the proof of Theorem\,\ref{asir2}.

{
  \begin{Lemma}\label{L42}
{(i)}    Let $\bfy(x)$ be a solution of \eqref{sysP1}-\eqref{noP1} so that $\bfy(x)=o(x^{-3})$ as $|x|\to\infty$ with $\arg(x)=a$ (for some $a$).

    Then $\bfy(x)$ has a formal asymptotic power series in powers of $x^{-1}$ and the asymptotics is differentiable.

(ii) For any $k\in\ZZ$ there is a unique solution of  \eqref{sysP1}-\eqref{noP1}

 such that
 \begin{equation}
   \label{eq:eqasyh}
   \bfy=o(x^{-2}\ln(x)^2)
 \end{equation}
as $|x|\to\infty$ with $\arg(x)=(2k+1)\pi i/2$, $k\in \ZZ$; these are tritronque\'ees.
  \end{Lemma}
  \begin{proof}
  (i) Take $a\in[-\tfrac{\pi}{2},\tfrac{\pi}{2}]$ (the proof for other $a$ is similar). Let $\tfrac{1}{x_0}$ and $\epsilon$ be small enough, where $\arg x_0=a$. We write  \eqref{sysP1} in the integral form
    \begin{align}
      \label{eq:systy1}
      y_1(x)={A_1}x^{-\frac12} e^{-x}+x^{-\frac12}e^{-x}\int_{x_0}^xs^{\frac12}e^s g_1(s^{-1},\bfy(s))ds\nonumber \\  y_2(x)={A_2}x^{-\frac12}e^{x}+x^{-\frac12}e^{x}\int_{\infty e^{i\phi}}^{x}s^{\frac12}e^{-s} g_2(s^{-1},\bfy(s))ds
    \end{align}
Since $\bfy=o(x^{-3})$ and from the properties of $\bf g$  we see that the second integral is  convergent and (again since $\bfy=o(x^{-3})$) we must have ${A}_2=0$. It is straightforward to show that \eqref{eq:systy1} is contractive in the ball of radius $\epsilon$ in  the norm $\|\bfy\|=\sup_{|x|>|x_0|}\|x^3 \bfy(|x|e^{ia})\|$, if $\epsilon$ and $1/|x_0|$ are small. It then follows that $\bfy ={\bf a}_4x^{-4}+o(x^{-4})$, where ${\bf a}_4=\frac{392}{625}(1,1)$; feeding back this estimate into \eqref{eq:systy1} it follows that $\bfy$ is of the form $\bfy ={\bf a}_4x^{-4}+{\bf a}_5x^{-5}+o(x^{-5})$, etc.
Differentiability of the asymptotics follows from the integral form \eqref{eq:systy1}.

(ii) We can w.l.o.g. take $k=0$.  The proof is similar to that of (i), with $O(x^{-2}\ln^2 x)$ replacing $o(x^{-3})$ and the lower limit of integration $x_0$ replaced by $\infty e^{i\pi/2}$.
  \end{proof}
}
\begin{Note}
  {\rm
The scale $x^{-2}\ln^2 x$ is chosen for technical reasons, since \cite{inprep} finds the asymptotics in the pole region up to  $O(x^{-2}\ln^2 x)$ errors.

}
\end{Note}

Next theorem establishes that generalized Borel summation of the transseries \eqref{tranyt}, and of \eqref{tranytt}, produces actual solutions, \eqref{BoSumTrans}, of the system (see also \S\ref{LBsum}).

\begin{Theorem}\label{Note43}
Consider the system \eqref{sysP1}-\eqref{noP1}.

I. Assume  $\arg x\in {[}-\tfrac{\pi}{2},\tfrac{\pi}{2}]$. There exist functions $\mathbf{Y}_k,\, (k\Ge 0)$ such that the following hold.

(i)  ${\bf Y}_k$ are the Borel transforms of $\tilde{\bfy}_k$ and for any $\phi\in [-\tfrac{\pi}{2},0)\cup(0,\tfrac{\pi}{2}{]}$ we have
\begin{equation}
  \label{eq:y_k}
 \mathcal{L}_\phi{\bf Y}_k \sim\tilde{\bfy}_k(x)\ \ \ \ \text{for }x\to\infty,\ \ x\in e^{-i\phi}\RR^+
\end{equation}

  {(ii)  Let $\phi{\in {[}-\tfrac{\pi}{2},\tfrac{\pi}{2}{]}}$. If $\mathbf{y}(x)$ solves (\ref{sysP1}) with
    $\mathbf{y}=o(x^{-3})\ \ \text{for }x\to\infty,\ x\in e^{-i\phi}{\RR^+}$ then  $\mathbf{y}$ has a {\em{unique expansion as a Borel
    summed transseries}}: for some constant $C$
\begin{equation}\label{BoSumTrans}
 {\mathbf{y}}(x;C)=\mathcal{L}_\phi{\bf Y}_0\, (x)\, +\sum_{k=1}^\infty C^ke^{-kx}\,\mathcal{L}_\phi{\bf Y}_k\, (x)\ \ \ \text{for }\ x\in e^{-i\phi}\RR^+,\ |x|\ \text{large}
 \end{equation}
If $\phi=\pm \tfrac{\pi}{2}$ then $C=0$.

(iii) The constant $C$ in \eqref{BoSumTrans} depends on the direction $\phi$, $C=C(\arg x)$, and is piecewise constant; it can only change at the Stokes direction {\rm{arg}}$\, x=0$.

(iv) $\mathbf{Y}_k$ have the following regularity properties:

\begin{enumerate}[(a)]
\item\label{itema}
 $\mathbf{Y}_0(p)=p^3\mathbf{A}_0(p)$, and, for $k\Ge 1$,
$\mathbf{Y}_k(p)=p^{k/2-1}\mathbf{A}_k(p)$, with $\mathbf A_k(p)$ analytic on the universal covering of
$\CC\setminus\{\pm1,\pm2,\ldots\}$; and $\mathbf A_1(0)$ normalized to equal $\mathbf{e}_1$. Also
\bel{Y0pe1}
\mathbf{Y}_0(p)=(1-p)^{-1/2}\,S_Y\,\mathbf{A}(p)+\mathbf{B}(p)
\ee
with $\mathbf{A}(p), \mathbf{B}(p)$ analytic at $p=1$, $\mathbf{A}(p)=\mathbf{e}_1+O(p-1)$, and $S_Y$ is a constant;
\item all $\mathbf{Y}_k(|p|e^{i\phi})$ are
left and right continuous in $\phi$ at $\phi=0$ and $\phi=\pi$ and are in $L^1_{loc}{(\RR_+)}$.

\item
For any $\delta>0$ there is a large enough $b$ so that $\|\mathbf{Y}_k\|_b<\delta^k$ for $k=0,1\ldots$ where $\|f\|_b:=\int_0^\infty e^{-bt}|f(te^{i\phi})|dt$. As a consequence, each $\mathbf{Y}_k$ is Laplace transformable along any direction of argument $\phi\in (0,\pi)\cup(\pi,2\pi)$.

\end{enumerate}
}

II. Analogously, for $x$ with $\arg x\in {[}\tfrac{\pi}{2},\tfrac{3\pi}{2}{]}$ (or $\arg x\in {[}-\tfrac{\pi}{2},-\tfrac{3\pi}{2}{]}$) similar  Borel
    summed transseries representations exist for solutions $\mathbf{y}(x)$ of (\ref{sysP1}) with
    $\mathbf{y}(x)=o(x^{-3})\ \ \text{for }x\to\infty,\ x\in e^{-i\phi}{\RR^+}$.
    The constant $C$ can only change at the Stokes directions {\rm{arg}}$\, x=\pm\pi$.

\end{Theorem}

\

{\em{Proof of Theorem\,\ref{Note43}.}} Part {\em{I.}} follows from general results in \cite{imrn}, \cite{duke}; part of the theorem is also found in  \cite{invent}.
More precisely, using Lemma \ref{L42}, Proposition\,\ref{Formh} and \eqref{eq:trsf3} it follows that $\tilde{\bfy}_0$ is unique, so that  {Theorem}\,\ref{Note43} (i)-(ii) follows from Theorem {2} (ii) in \cite{imrn}, while {Theorem}\,\ref{Note43} (iii)  (a),\,(b) follows from Proposition\, 1 in \cite{imrn} and (c) from Proposition\,20 in \cite{imrn}. The fact that $\mathbf{Y}_0(p)=O(p^3)$ as $p\to 0$ follows from $\tilde{\bfy}_0=O(x^{-4})$ (by \eqref{defhh0}, \eqref{eq:trsf3}) and Proposition\, 1\,ii) in \cite{imrn}.

We note that along Stokes directions, when arg$\, x$ is $0$ or $\pm\pi$, the Laplace transform does not exist as a usual integral, and must be considered in a generalized  sense \cite{duke}.

Part {\em{II.}} follows due to the symmetry of \eqref{eq:eqp} under $h(x)\mapsto h(xe^{
      \pm i\pi})$.
$\Box$

\subsubsection{Properties of $H_k$}

\begin{Lemma}\label{Lemma2}
Let $H_k$ be as in \eqref{defHk}.
Then
\begin{equation}\label{asyHk}
\mathcal{L}_\phi H_k\,(x)\sim \tilde{h}_{k}(x); \ \ \ \tilde{h}_{k}(x)=O(x^{-k/2})\ \ \text{for }x\to\infty,\,\arg x\in
(-\tfrac{\pi}{2},\tfrac{\pi}{2})
\end{equation}
and the analytic continuation of $H_k$ have the following regularity properties.

(a) The function $H_0(p)$ satisfies the inverse Laplace transformed equation
\eqref{eq:eqp}:
\begin{equation}
  \label{eq:Bplane}
  (p^2-1)H=\frac{196}{1875}\, p^3+p*H+\frac{1}{2}H*H
\end{equation}
and $H_0$ is the unique solution of \eqref{eq:Bplane} which is analytic at $p=0$.

(b) $H_0$ is an odd function.

(c) ${H}_0(p)=p^3{A_0}(p)$ with $A_0(p)$ analytic on the universal covering of
$\CC\setminus\{\pm1,\pm2,\ldots\}$;

(d) $H_0$ satisfies the statement of Proposition\,\ref{H0_behave}(i), and therefore also ${H}_0(p)=-S\,(1+p)^{-1/2}{A}(-p)-{B}(-p)$ for the same   $S,A,B$ ; also, for $k\geq 1$, $H_k(p)=p^{k/2-1}A_k(p)+B_k(p)$ with $A_k,B_k$ analytic at $p=0$ and $A_1(0)=1$.

(e) $H_0(|p|e^{i\phi})$ is
left and right continuous in $\phi$ at $\phi=0$ and $\phi=\pi$, and is in $L^1_{loc}{(\RR_+)}$.

(f) For any $\delta>0$ there is a large enough $b$ so that $\|H_k\|_b<\delta^k$ for $k=0,1\ldots$ where $\|f\|_b:=\int_0^\infty e^{-bt}|f(te^{i\phi})|dt$. As a consequence, each $H_k$ is Laplace transformable along any direction of argument $\phi\in (0,\pi)\cup(\pi,2\pi)$.

\end{Lemma}

 {\em{Proof.}}

 Solutions $h(x)$ are obtained from solutions of \eqref{p1} via \eqref{chvarx}. Using the known asymptotic forms of solutions of  \eqref{p1}, see \cite{Clarkson}, it follows that \eqref{assph} implies that $h(x)=O(x^{-4})$ in the same sector, hence $\bfy$ given by \eqref{eq:trsf3} is $o(x^{-3})$.

Note that we have, in view of \eqref{htoy},
\bel{like28}
\tilde{h}_k=\frac{1}{2}\left(1-\frac{1}{4x}\right)\tilde{y}_{k;1}+\frac{1}{2}\left(1+\frac{1}{4x}\right)\tilde{y}_{k;2}
\ee
(where $\tilde{\bfy}_{k}=(\tilde{y}_{k;1},\tilde{y}_{k;2})^T$). We have (see \eqref{xLBxm})
\bel{like29}
\mathcal{L}_\phi H_k\, (x)=\mathcal{L}_\phi\mathcal{B}\tilde{h}_k\,(x)=\frac{1}{2}\left(1-\frac{1}{4x}\right)y_{k;1}(x)+\frac{1}{2}\left(1+\frac{1}{4x}\right)y_{k;2}(x)=h_k(x)
\ee
Using Theorem \ref{Note43} this implies \eqref{asyHk}.

Relation \eqref{eq:Bplane} follows from the fact that $H_0=\mathcal{B}\tilde{h}_0$, hence it satisfies the inverse Laplace transformed equation \eqref{eq:eqp}.

(b) Equation \eqref{eq:Bplane} has a unique solution analytic at $p=0$ \cite{imrn} and it can be easily checked that if $H(p)$ is a solution, then so is $-H(-p)$.

To prove the other properties, note that we have, using \eqref{like28}, and the fact that $\mathcal{B}(x^{-1}\tilde{\mathbf{Y}}_k)=1*\mathcal{B}(\tilde{\mathbf{Y}}_k)$,
\begin{multline}\label{e42}
H_k=\mathcal{B}\tilde{h}_k=\frac{1}{2}\left(Y_{k;1}-\frac{1}{4}\,1*Y_{k;1}\right)+\frac{1}{2}\left(Y_{k;2}+\frac{1}{4}\,1*Y_{k;2}\right)\\
=\frac{1}{2}\left(Y_{k;1}+Y_{k;2}\right)+\frac{1}{8}\int_0^p \left(-Y_{k;1}(q)+Y_{k;2}(q)\right)\, dq
\end{multline}
and the properties follow from {Theorem}\,\ref{Note43}.
$\Box$

\subsubsection{Proof of Theorem\,\ref{asir2}}
I. follows from Lemma\,\ref{Lemma2}, Theorem\,\ref{Note43} and \eqref{like28}-\eqref{e42}.

II. follows due to the symmetry of \eqref{eq:eqp} under $h(x)\mapsto h(xe^{
      \pm i\pi})$.

\subsection{Proof of Proposition\,\ref{H0_behave}}\label{PfP3}

These statements are true in a general setting, see \cite{imrn}; see also \S\ref{Sbeta} for an overview in the general case and more details. The main ideas are as follows.

(i) was proved in Lemma\,\ref{Lemma2}.

(ii) By rotating the angle $\phi$ into $\phi\in[-\pi, -\frac{\pi}{2}]$, and using the estimates of Lemma\,\ref{Lemma2}(c),(d)
it is clear that $h^+(x)$ can be analytically continued for $x$ in the second quadrant. Continuation of $h^+(x)$ for $x$ in the fourth quadrant is done by deformation of the path of integration $\arg p=\phi<0$ to a direction with $\arg p>0$ plus an infinite sum of paths coming from $\infty$ in the first quadrant, encircling only one point $k\in\ZZ_+$ counterclockwise, and going back to $\infty$. The series obtained converges due to Lemma\,\ref{Lemma2}(d) and we obtain a Borel summed transseries for this continuation of $h^+(x)$. The analytic continuation of $h^-(x)$ is similar.

Formula \eqref{jumph01} is proved below, but it is more general, see \S\ref{Sbeta}
 in the Appendix.

We have
 $h^+(x)-h^-(x)=\int_\ell e^{-xp}H_0(p)\, dp$ where $\ell$ is a path coming from $+\infty$ above $[1,+\infty)$, going counterclockwise around $p=1$ and returning to $+\infty$ below $[1,+\infty)$. Then, choosing the usual branch of the radical (with $(1-p)^{1/2}>0$ for $p<1$) we have
 \begin{multline}
 h^+(x)-h^-(x)\sim S\int_\ell e^{-xp}(1-p)^{-1/2}\, dp\\
 =S\int_{+\infty}^1i(p-1)^{-1/2}e^{-xp}\, dp+S\int_1^{+\infty}(-i)(p-1)^{-1/2}e^{-xp}\, dp\\
 =-2i\sqrt{\pi}Se^{-x}x^{-1/2}
 \end{multline}

(iii) Similar arguments for $\mathcal{L}_{\phi}H_1\,(x)$ show that the analytic continuation produces only terms of order $e^{-2x}$ or smaller, see Theorem\,\ref{Note43}(iv)(a), \eqref{like28},\,\eqref{like29}.
$\Box$

\subsection{Specification of branches and proof of Theorem\,\ref{TritCis0}}\label{PfTritCis0}

\subsubsection{Branches}\label{brachoise} We recall that $H_0(p)$ is analytic at $p=0$, see Lemma\,\ref{Lemma2}(c). Each directional Laplace transform of $H_0$ uses the analytic continuation of this germ of analytic function at $p=0$ along the direction of the transform.

$H_k(p)$ with $k\geq 1$ may have a square root branch point at $p=0$, see Lemma\,\ref{Lemma2}(d). We use the analytic continuation of the usual branch of the square root, with $\arg p=0$ for $p>0$; for instance in the integral in the third  expansion in \eqref{ht_trans}, the functions are continued through clockwise rotation, starting with $\arg p=0$.

\subsubsection{Proof of Theorem\,\ref{TritCis0}}
(i) By Theorem\,\ref{asir2} a tritronqu\'ee \eqref{sectorinz} has a series representation of the form \eqref{eq:transh} for some $C_+$.
This $C_+$ must be zero, otherwise $h_t$ has poles beyond $\arg x=\tfrac{\pi}{2}$, by Proposition \ref{RegP}, hence it does not correspond to the tritronqu\'ee \eqref{sectorinz}.

 (ii) Note that for $x$ in the second quadrant ($\arg x=\pi-\epsilon$), $h^+=\mathcal{L}_\phi H_0$ with $\phi=-\pi+\epsilon$ and that after clockwise rotation in the $p$-plane we have  $(1+p)^{1/2}=i |1+p|^{1/2}$ for $p<-1$.

We have $h^+(x)-\hsig(x)=\int_{\tilde{\ell}} e^{-xp}H_0(p)\, dp$ where $\tilde{\ell}$ is a path coming from $-\infty$ above $(-\infty,-1]$, going clockwise around $p=-1$ and returning to $-\infty$ below $(-\infty,-1]$. Then
 \begin{equation}
 h^+(x)-\hsig(x)\sim
 2i\int_{-1}^{-\infty} |p+1|^{-1/2} e^{-xp}(-S)\, dp=2iS\sqrt{\pi}e^x|x|^{-1/2}
 \end{equation}
 which gives \eqref{jumphpi} (noting that analytic continuation in $x$ is counterclockwise).

 (iii) Formula \eqref{ht_trans} for $|\arg x|\Le\tfrac{\pi}{2}$ follows by \eqref{eq:transh}, since $C_+=0$ by (i) and  by Proposition\,\ref{H0_behave}(iii).

 Then, for $\arg x\in [\tfrac{\pi}{2},\pi)$, in $h^+=\mathcal{L}_\phi H_0$ can be analytically continued  by rotating of $\phi$, since $H_0$ is analytic and Laplace transformable by Lemma\,\ref{Lemma2}(c),(f).

To obtain the series \eqref{ht_trans} for $\arg x\in (\pi, \tfrac{3\pi}{2})$ we use (ii) and the proof is analogous to the proof of  Proposition\,\ref{H0_behave} (ii), only the branches are different; alternatively, this follows from the symmetry $h(x)\to\overline{h(-\overline{x})}$.

\subsection{Proof of Theorem\,\ref{FAPIn}}\label{PfRem8}

 (i) is proved in \cite{invent}. We include the main steps in the calculation of $F_n$ in \S\ref{CalcHn}.

 (ii) It is straightforward to check that the transformation
  $g=h(1+\frac13h)^{-1}$ leads to a system of equations satisfying the same assumptions as  \eqref{sysP1}, and the construction of the expansion  \eqref{eq:devh1} mirrors the one for \eqref{eq:eq51}. Expectedly, the expansion in \eqref{eq:devh1} coincides with the formal asymptotic expansion of $h(1+\frac13h)^{-1}$ in powers of $1/x$ using \eqref{eq:eq51}.

\subsection{Proof of Proposition \ref{RegP}}\label{PfPr10}
The fact that this is the first array of poles is guaranteed by \eqref{eq:eq51} which shows that for $|\xi|<a<12$, $h$ is bounded.

{(i)  Let $g$ be as in Theorem\,\eqref{FAPIn}(ii). We first note that a pole of $h$ is a regular point of $g$, one in which $g$ assumes the value $3$.
Formula \eqref{eq:pospoles} is simply obtained from the implicit function theorem and the expansion of $g$ in
Theorem \cite{invent} (ii), by solving the implicit equation $g(x)-3=0$, writing $\ln\xi\sim \ln 12+2 n\pi i...$ and  retaining three orders in $n^{-1}$ in the calculation.}

{(ii) The roots $r_{1,2}$  of the quadratic $H_0(r)=y$ satisfy $r_1 r_2=12$. Thus
$H_0$ maps the disk $\{z:|z|\le 12\}$ onto the Riemann sphere $\CC\cup\{\infty\}$. If $\mathcal{A}(0)=0,\mathcal{A}'(0)=1$ and $\mathcal{A}$ is analytic, it follows that  $\mathcal{A}(h(x))$ is singular in $\{z:|z|\le 12\}$. The rest is immediate.}

\section{Appendices}\label{Appendix}

Sections \S\ref{IntroTran}-\ref{Sbeta} contain an outline of some results found in \cite{imrn},\,\cite{duke} and illustration of these results on some simple examples. Section \S\ref{historic} contains a brief overview on the development of the subject.

\subsection{Representation of solutions as transseries}\label{IntroTran}

Very few differential equations can be explicitly solved, and even when this is possible, their expression may be too complicated for easily extracting useful information about solutions; by contrast, formal solutions can often be obtained algorithmically as asymptotic expansions, from which properties of solutions such as rate of decay/increase, or approximations, can be easily read. Sometimes, free parameters are "hidden" beyond all orders of a classical asymptotic series; in such cases transseries are instrumental in uncovering these parameters.

It is well known that equations written in terms of analytic functions have convergent power series solutions at any regular point of the equation; convergent expansions for solutions also exist at regular singularities (generically). But at irregular singularities the asymptotic expansions of solutions are often divergent.

To illustrate, consider the simple equation
\bel{eq1}
y'+y=\frac{1}{x^2}
\ee
The point $x=\infty$ is an irregular singular point of the equation (indeed, the substitution $x=1/z$ maps  $\infty$ to $0$  and  brings \eqref{eq1} to the form $-z^2\tfrac{dy}{dz}+y=z^2$ for which $z=0$ is an irregular singularity).

It is easy to see that there exists a unique asymptotic power series solution for $x\to\infty$,
\bel{soleq0}
\tilde{y}_0(x)=\sum_{n=2}^\infty\frac{(n-1)!}{x^n}
\ee
and it is divergent. On the other hand the general solution of \eqref{eq1} is
\bel{soleq1}
y(x;C)=y_0(x)+Ce^{-x},\ \ \ \text{where\ \  }y_0(x)=e^{-x}\int_{x_1}^x\frac{e^s}{s^2}\, ds\,\sim\,\tilde{y}_0(x)\ \text{as\ }x\to+\infty
\ee
and $C$ is a free parameter.

This simple example illustrates main phenomena at irregular singularities: the power series solutions are divergent, there is loss of information (in \eqref{soleq1} there is a one parameter family of solutions asymptotic to the same power series), and asymptoticity holds only in sectors ($y_0(x)\sim \tilde{y}_0$ only for $x\to\infty$ with $|\arg x|<\tfrac{\pi}{2}$).

In view of \eqref{soleq0},\eqref{soleq1} it is natural to
consider that the {\em{complete formal solution}} of \eqref{eq1} is
\bel{hsoleq1}
\tilde{y}(x)=\tilde{y}_0(x)+Ce^{-x},\ \ \text{for }\ x\to +\infty
\ee
The formal expression \eqref{hsoleq1}, which satisfies \eqref{eq1}, is not an asymptotic series in the sense of Poincar\'e if $|\arg x|<\tfrac{\pi}{2}$, as the term $Ce^{-x}$ is much smaller than all the powers of $x$ in $\tilde{y}_0(x)$: it is a {\em{term beyond all orders}} of the main series. The formal solution \eqref{hsoleq1} is the simplest example of a {\em{transseries}}.

Consider next a nonlinear example, namely \eqref{eq1} plus a nonlinear term:
\bel{neq1}
y'+y=\frac{1}{x^2}+y^4
\ee
Again, equation \eqref{neq1} has a unique power series solution
$\tilde{y}_0(x)=\frac{1}{x^2}+\frac{2}{x^3}+\frac{6}{x^4}+\ldots$ which can be shown to be divergent too. To find possible further terms in a formal expansion we search for a perturbation: substituting $y=\tilde{y}_0+\delta$ in \eqref{neq1} and using the fact that $\tilde{y}_0$ is already a formal solution we get $\delta'+\delta\sim 2\tilde{y}_0\delta$, giving $\delta\sim Ce^{-x}\tilde{y}_1(x)$ where $\tilde{y}_1(x)$ is a (divergent) power series. Since $\tilde{y}_0+Ce^{-x}\tilde{y}_1(x)$ does not solve the equation, further corrections are required, yielding a complete formal solution of \eqref{neq1} in the form
\bel{hnsoleq1}
\tilde{y}=\tilde{y}(x;C)=\, \tilde{y}_0(x)    +C{\rm{e}}^{-x}\tilde{y}_1(x) +  C^2{\rm{e}}^{-2x}\tilde{y}_2(x)+\ldots
\ee
where $\tilde{y}_k(x)$ are divergent power series and $C$ is an arbitrary parameter. The formal solution \eqref{hnsoleq1} is a transseries for $x\to\infty$ along directions in the complex plane for which the terms can be well ordered decreasingly, namely for $x\in e^{ia}\RR_+$ with $ |a|<\tfrac{\pi}{2}$.

Scalar equations more general than \eqref{neq1}, of the form
\bel{gneq1}
y'+\left(\lambda-\frac{\alpha}{x}\right)y=g(x^{-1},y)\ \ \text{with } g=O(x^{-2})+O(y^2)\ \ \text{for } x\to\infty,\, y\to 0
\ee
have formal series solution of the form
\bel{gnsoleq1}
\tilde{y}=\tilde{y}(x;C)=\,\tilde{y}_0(x)+\sum_ {k=1}^\infty C^k{\rm{e}}^{-\lambda kx}\tilde{y}_k(x),\ \ \text{with }\ \tilde{y}_k(x)=x^{k\alpha}\tilde{s}_k(x)
\ee
(with $\tilde{s}_k(x)$ an integer power series in $x^{-1}$) which is a transseries for $x\to\infty$ along any direction for which $|\arg (\lambda x)|<\tfrac{\pi}{2}$.

Systems have transseries solutions which are similar: generic equations can be brought to the normal form (in \cite{imrn} this is eq. (1.1))
\bel{normform}
\begin{array}{l}
\displaystyle{ \bfy'+\left(\Lambda-\frac{1}{x}A\right)\bfy=\mathbf{g}(x^{-1},\bfy)}\\
\ \\
\displaystyle{ \text{with }\mathbf{g}=O(x^{-2})+O(|\bfy|^2)\ \ \text{for } x\to\infty,\, |\bfy|\to 0}\\
\ \\
\displaystyle{\text{and }\Lambda=\rm{diag}(\lambda_1,\ldots,\lambda_d),\ A=\rm{diag}(\alpha_1,\ldots,\alpha_d)}
\end{array}
\ee
Under appropriate nonresonance conditions\footnote{It is required that $\lambda_1,\ldots,\lambda_d$ be linearly independent over $\ZZ$ (otherwise its expression has a slightly more general form) and it suffices that the Stokes lines be distinct.} systems  \eqref{normform} have formal solutions
\bel{gtrans}
\tilde{\bfy}=\tilde{\bfy}(x;\mathbf{C})=\,\tilde{\bfy}_\mathbf{0}(x)+\sum_ {\mathbf{k}\in\NN^d\setminus\mathbf{0}} \mathbf{C}^\mathbf{k}{\rm{e}}^{-\boldsymbol\lambda \cdot \mathbf{k}x} \tilde{\bfy}_\mathbf{k}(x)   \ \ \text{with }\tilde{\bfy}_\mathbf{k}(x)=x^{\boldsymbol\alpha\cdot \mathbf{k}}\tilde{\bfs}_\mathbf{k}(x)
\ee
which is a transseries for $x\to\infty$ along any direction along which the terms can be well ordered, meaning that all the exponentials are decaying, therefore along any direction in the sector
\bel{Strans}
\displaystyle{S_{trans}=\{x\in\CC\, |\, \Re(\lam_j x)>0\ {\rm{for\ all\,}}j\ {\rm{with\ }}C_j\ne 0\} }
\ee

\subsection{Correspondence between transseries and actual solutions: generalized Borel summation}\label{LBsum}

Consider the linear equation \eqref{eq1}. Since its formal solution \eqref{soleq0} is factorially divergent and $\mathcal{L}(p^{n-1})=(n-1)!x^{-n}$, heuristically, it is natural to attempt to write $\tilde{y}_0$ as a Laplace transform: this is central to Borel summation.

Recall that the Borel transform is defined as the formal inverse Laplace transform,
$\mathcal{B}(x^{-\alpha})=\tfrac{p^{\alpha-1}}{\Gamma(\alpha)}$ for $\alpha>0$ (where $\mathcal{L}$ is the Laplace transform), and, more generally, we have \eqref{DefB}.

Taking the inverse Laplace transform, equation \eqref{eq1} becomes
 $(1-p)Y(p)=p$ therefore
 \bel{Ltranex1}
 y=\mathcal{L}_\phi Y\ \ \ \ \text{where }Y(p)=\frac{p}{1-p}
 \ee
and $\mathcal{L}_\phi$ is the Laplace transform along a direction of argument $\phi$, see \eqref{DefLphi}.

Note that we cannot take the Laplace transform along $\RR_+$ in \eqref{Ltranex1}(except in the sense of distributions \cite{duke}), but
we can integrate on any half-lines above, or below $\RR_+$, obtaining
$$y_0^+(x)=\mathcal{L}_\phi Y(x)\ \ \ \ \text{for }-\phi=\arg x\in(0,\frac{\pi}{2})$$
and
$$y_0^-(x)=\mathcal{L}_\phi Y(x)\ \ \ \ \text{for }-\phi=\arg x\in(-\frac{\pi}{2},0)$$
The values of $y_0^\pm(x)$, do not depend on the value of $\phi$ in its specific quadrant; they can be both analytically continued in the right half-plane (and beyond) $ y_0^+(x)\ne y_0^-(x)$ and in fact the difference $ \tfrac{1}{2\pi i}(y_0^+(x)- y_0^-(x))= e^{-x}$ recovers the exponentially small term in \eqref{soleq1}.

These facts generalize to nonlinear equations. Consider the example \eqref{neq1};
taking the inverse Laplace transform one obtains the convolution equation
\bel{Lneq}
(1-p)Y(p)=p+Y^{*4}(p)
\ee
which has a unique solution $Y=Y_0(p)$ analytic at $p=0$. In fact, $Y_0(p)$ is analytic for $|p|<1$, and it is singular at $p=1$.
Due to the convolution term  in \eqref{Lneq},the singularity at $p=1$ gives rise to  an equally spaced array of singularities in the Borel plane at $p=2,3,4,\ldots$. $Y_0(p)$ is analytic along any direction $p=|p|e^{i\phi}$ with $0<|\phi|<\tfrac{\pi}{2}$, is Laplace transformable, and
$y_0=\mathcal{L}_\phi Y_0=\mathcal{L}_\phi\mathcal{B}\tilde{y}_0$ is an actual solution of  \eqref{neq1} for $x$ large with $\arg x=-\phi$: $y_0(x)$ is the Borel summation, along the direction of $x$, of $\tilde{y}_0(x)$.

The Borel sum  $y_0^+=\mathcal{L}_\phi\mathcal{B}\tilde{y}_0$ of $\tilde{y_0}$ is the same for all $-\phi=\arg x\in (0,\tfrac{\pi}{2})$ and $y_0^-=\mathcal{L}_\phi\mathcal{B}\tilde{y}_0$ is the same for all $-\phi=\arg x\in (-\tfrac{\pi}{2},0)$, both $y_0^\pm$ can be analytically continued in the right  half-plane, and $y_0^+-y_0^-$ is exponentially small.

The other series $\tilde{y}_k$ in \eqref{hnsoleq1} are Borel summed similarly (using convolutions equations which are found for $Y_k=\mathcal{B}\tilde{y}_k$ in \cite{imrn},\,\cite{duke}), yielding functions
$y_k=\mathcal{L}_\phi Y_k=\mathcal{L}_\phi\mathcal{B}\tilde{y}_k$ analytic for large $x$; the series
$$y_0(x)+\sum_{k=1}^\infty C^ke^{-kx}y_k(x)$$
converges for $x$ large with $\arg x=-\phi\in(-\tfrac{\pi}{2},0)\cup(0,\tfrac{\pi}{2})$ to a solution of \eqref{neq1}\footnote{For $\arg x=0$ solutions are obtained using weighted averages of Laplace transforms along paths going toward $\infty$ avoiding the singularities in prescribed ways, independent of the equation.} \cite{imrn},\,\cite{duke}.

The general one-dimensional case \eqref{gneq1} is similar; $Y_0(p)$ will have equally spaced arrays of singularities along $\arg p=\arg \lambda$. Along any other direction $p=|p|e^{i\phi}$ with $0<|\phi-\arg\lambda|<\tfrac{\pi}{2}$ the continuation of $Y_0(p)$ is analytic (generally on a Riemann surface) and Laplace transformable, and $\mathcal{L}_\phi Y_0=\mathcal{L}_\phi\mathcal{B}\tilde{y}_0$ is an actual solution of  \eqref{gneq1} for $x$ large with $\arg x=-\phi$.

{\bf{Remark.}} {\em To Borel sum the series in \eqref{gnsoleq1} for $k\Ge 1$ we may consider $y_k=\mathcal{L}_\phi\mathcal{B}\tilde{y}_k$ (if $\alpha<0$), or we can choose any $m$ (large enough so that $\alpha-m<0$), find solutions as Borel summed transseries in the form
\bel{Bstm}
 y_0(x)+\sum_{k=1}^\infty C^ke^{-kx}x^{mk}\mathcal{L}_\phi\mathcal{B}(x^{-mk}\tilde{y}_k)
 \ee
The final result does not depend on $m$, since
\bel{xLBxm}
x^N\mathcal{L}_\phi\mathcal{B}(x^{-N}x^{-n})=\mathcal{L}_\phi\mathcal{B}(x^{-n})
\ee
}

\z Finally, {the series} \eqref{Bstm} {converges} to actual solutions for $x\in S_{an}$ where
$$\displaystyle{ S_{an}=\{x\, \big|\, -\frac{\pi}{2}+\epsilon<\arg(x)<\frac{\pi}{2}-\epsilon,\ |x|>R\} }$$

Generic nonlinear equations \eqref{normform} have their transseries solutions \eqref{gtrans} summed similarly along directions $d$ in $\CC$. Furthermore, there exists a one-to-one correspondence between
 solutions s.t. $\bfy(x)\to 0$ ($ x\in d,\,x\to\infty$) and
 (generalized) Borel sums of $\tilde{\bfy}(x;C)$ transseries solutions along $d$.
  These solutions $\bfy(x;C)$ are analytic in a sector containing $d$ for $|x|$ large. These results are stated and proved in \cite{imrn}, \cite{duke}.

Theorem\,\ref{Note43} is an application of these results for the system (\ref{sysP1}) associated to the Painlev\'e equation \P1.

\subsection{The Stokes phenomenon}\label{Sbeta}

The directions $\pm i\overline{\lam}_j\RR_+$ are called {\em{antistokes lines}}; along these directions, some exponential $e^{-\lambda_jx}$ in \eqref{gtrans} is purely oscillatory. Antistokes directions border the sectors where transseries exist, \eqref{Strans}. Directions with $\lambda_jx\in\RR_+$ (for some $j$) are called {\em{Stokes lines}}; along these, some exponential $e^{-\lambda_jx}$ has fastest decay. At Stokes directions the constants beyond all order in the one-to-one association between small solutions and transseries may change: this is the {\em{Stokes phenomenon}}.

To illustrate this consider \eqref{eq1}. As noted above, solutions can be written using $Y_0$, the Borel sum of $\tilde{y}_0$ as
 \begin{equation}\label{eq:try}
    y(x)=\left\{
    \begin{array}{l}
    \mathcal{L}_\phi Y_0\,(x)+ C_+ e^{-x}\ \ \ \text{for }
   - \phi= \arg x\in(0,\tfrac{\pi}{2})\\
    \mathcal{L}_\phi Y_0\,(x)+ C_- e^{-x}\ \ \ \text{for }
   - \phi= \arg x\in(-\tfrac{\pi}{2},0)
    \end{array}\right.
  \end{equation}

The value of the jump in the constant beyond all orders, $C_+ -C_-$, is called the {\em{Stokes constant}}.

More generally, a fixed solution of \eqref{gneq1} can be written as Borel summed transseries \eqref{Bstm} for some fixed $C$ for all $\phi$ with $\arg\phi\in\arg\lambda+(0,\tfrac{\pi}{2})$, and with a different $C$ for all $\phi$ with $\arg\phi\in\arg\lambda+(-\tfrac{\pi}{2},0)$.

 For general equations the situation is similar: the vector parameter $\mathbf{C}$ in a transseries \eqref{gtrans} associated via Borel summation  along a direction to a true solution does not change when this direction varies between two consecutive Stokes or antistokes lines, but it generally changes across a  Stokes line.

 Consider systems \eqref{normform}, with $\lambda_1=1$, $|\lambda_j|\Ge 1$ and $\beta:=\beta_1<1$ (which can be arranged by a suitable substitution) and solutions obtained by Borel summation of the transseries solution \eqref{gtrans} along directions slightly above and below the Stokes line $\arg x=0$:
 \begin{equation}\label{GenTrans}
    \bfy(x)=\left\{
    \begin{array}{l}
    \mathcal{L}_\phi \mathbf{Y}_\mathbf{0}\,(x)+ \sum_ {\mathbf{k}\in\NN^d\setminus\mathbf{0}} \mathbf{C_+}^\mathbf{k}{\rm{e}}^{-\boldsymbol\lambda \cdot \mathbf{k}x}  \mathcal{L}_\phi\mathbf{Y}_\mathbf{k}(x)   \ \ \text{for }
   - \phi= \arg x\in(0,a_2)\\
    \mathcal{L}_\phi \mathbf{Y}_\mathbf{0}\,(x)+ \sum_ {\mathbf{k}\in\NN^d\setminus\mathbf{0}} \mathbf{C_-}^\mathbf{k}{\rm{e}}^{-\boldsymbol\lambda \cdot \mathbf{k}x}  \mathcal{L}_\phi\mathbf{Y}_\mathbf{k}(x) \ \ \ \text{for }
   - \phi= \arg x\in(a_1,0)
    \end{array}\right.
  \end{equation}
where $\mathbf{Y}_\mathbf{k}=\mathcal{B}_\phi\tilde{\bfy}_\mathbf{k}$ (is the analytic continuation of the Borel transform of $\tilde{\bfy}_\mathbf{k}$ along the direction of argument $\phi$), and the sector $a_1<\arg x<a_2$ does not contain another
 Stokes or antistokes line besides $\arg x=0$.

The first component $C_1$ of the constant beyond all orders in \eqref{GenTrans} changes when $\arg x$ crosses the Stokes line $\arg x=0$, corresponding to $\lambda_1=1$ \cite{duke}.

Changes in the constant beyond all orders occur upon analytic continuation across*  a Stokes line; the leading order change, which is exponentially small, is due to the continuation of $\mathcal{L}_\phi\mathbf{Y}_\bfk$.
The continuations of $\mathcal{L}_\phi\mathbf{Y}_\bfk$  generally add further, but these are of order $e^{-x}$ or smaller, and for $|\bfk|\Ge 1$, the $\mathcal{L}_\phi\mathbf{Y}_\bfk$ already multiplies an exponential, so this change does not affect the coefficient of $e^{-x}$. The fact that the changes in all $\mathbf{C}^\bfk$ with $|\bfk|\Ge 1$ match to give an overall jump equivalent to  $\mathbf C_+\to \mathbf C_-$ is due to the so-called {\em{resurgence}}, which links the singularities of all $\mathbf{Y}_\bfk$ in a precise manner.)

\subsubsection{The Stokes multiplier}
A calculation analogous to the one in the proof of Proposition\,\ref{H0_behave} gives the change in $C_1$, and the argument is as follows.
To analytically continue $\mathcal{L}_{0^-}\mathbf{Y}_0(x)$ past $\arg x=0$ we write $\mathcal{L}_{0^-}\mathbf{Y}_0=\mathcal{L}_{0^+}\mathbf{Y}_0+\delta$ where $\delta=(\mathcal{L}_{0^-}-\mathcal{L}_{0^+})\mathbf{Y}_0$. Since $\mathbf{Y}_0(p)$ is analytic for $|p|<1$ (by \cite{imrn} Proposition 1), the path of integration in $\delta$ can be deformed to the path from $\infty$ to $1$ below $[1,+\infty)$, going around $1$ and then going  to $+\infty$ above $[1,+\infty)$.

Using the fact that $\mathbf{Y}_0(p)=S_{\beta}(1-p)^{\beta-1} (\mathbf{e}_1+o(1))$ (by \cite{imrn} Proposition 1), and that $(1-p)^{\beta-1}=e^{\mp i\pi(\beta-1)}|1-p|^{\beta-1}$ for $|p|>1, \arg(p)=\pm 0$ we obtain, using Watson's Lemma,
$$\delta= -2i S_{\beta} \sin (\pi \beta) \int_1^\infty |1-p|^{\beta-1}e^{-px}dp (\mathbf{e}_1+o(1))=-2iS_{\beta} \sin (\pi \beta)\Gamma(\beta) e^{-x}x^{-\beta} (\mathbf{e}_1+o(1))$$
so that the jump in the constant $C_1$ across the Stokes line $\arg x=0$ is\footnote{We note that in \cite{imrn}, the constants $C(0\pm)$ correspond here to $C_\mp$, that in the formula below
  (1.15), the factor $\Gamma(\beta)$ is missing, in (1.19) $S_\beta$ should be $S$, and in (1.2) $\beta=\hat{B}_{1,1}$ with $\Re \beta\in (0,1]$.}
\begin{equation}
  \label{eq:sjump}
C_{1;+}-C_{1;-}=-S=-\mu\ \text{with}\  S=2iS_{\beta} \sin (\pi \beta)\Gamma(\beta),\ \ \beta=\beta_1;\ \ \ \Re(\beta)\in (0,1)
\end{equation}

\

For  general equations, the values of Stokes constants are transcendental.

\begin{Note}\label{SpecDir} The five special directions of \P1 are Stokes or antistokes lines of its normalized form
\eqref{eq:eqp}.
 \end{Note}

 \subsection{Further references.}\label{historic}      Double expansions of solutions of linear equations as power series multiplying exponentials have been studied starting with Fabry  \cite{Fabry} (1885), and then Cope \cite{Cope} (1936).
 Iwano  (1957-'59) analyzed  solutions of nonlinear systems as a convergent series of functions analytic in sectors, multiplying exponentials \cite{Iwano}. The subject has been developed and expanded substantially after the fundamental work of Ecalle (1981), with results in multisummability of power series of linear ODEs \cite{BBRS}, nonlinear ones, \cite{Bra}, transseries of nonlinear ODEs \cite{imrn},\cite{duke},  similar results for discrete equations \cite{Bra_dis},\cite{Bra_dis2}  and for PDEs \cite{OC-ST-PDE}; singularity formation near antistokes lines was studied in a general setting in \cite{invent}.

\subsection{Rewriting \eqref{eq:eqp} as a normalized system}\label{rewrP1}

We write \eqref{eq:eqp} as usual,
\begin{equation}
  \label{eq:eqsys1}
  \begin{pmatrix}
    h\\ h'
  \end{pmatrix}'=
  \begin{pmatrix}
    0\\\tfrac{392}{625}x^{-4}
  \end{pmatrix}+
  \begin{pmatrix}
    0 & 1\\1 &0
  \end{pmatrix}
\begin{pmatrix}
    h\\ h'
  \end{pmatrix}+
  \begin{pmatrix}
    0 & 0 \\ 0 & -\tfrac{1}{x}
  \end{pmatrix}\begin{pmatrix}
    h\\ h'
  \end{pmatrix}+
  \begin{pmatrix}
    0\\\tfrac12 h^2
  \end{pmatrix}
\end{equation}
The transformation
\begin{equation}
  \label{eq:trsf3}
  \begin{pmatrix}
    h\\ h'
  \end{pmatrix}=
   \frac{1}{2}\begin{pmatrix}
   1-\tfrac{1}{4x} & 1+\tfrac{1}{4x}\\ -1-\tfrac{1}{4x} &  1-\tfrac{1}{4x}
  \end{pmatrix}
  \begin{pmatrix}
    y_1\\y_2
  \end{pmatrix}
\end{equation}
brings (\ref{eq:eqsys1}) to \eqref{sysP1},\,\eqref{noP1}, which is in
the normal form \eqref{normform}.

 More precisely,
\begin{multline}\nonumber
g_1(x,\bfy)=-{\frac {1568}{625}}\,{\frac {4\,x+1}{ \left( 16\,{x}^{2}+1 \right) {x
}^{3}}}-\frac{1}{16}\,{\frac { \left( 4\,x-1 \right)  \left( 4\,x+1 \right) ^{
2}{y_1} { y_2}  }{ \left( 16\,{x
}^{2}+1 \right) x}}\\
-\frac{1}{32}\,{\frac { \left( 4\,x+1 \right)  \left( 4\,x-
1 \right) ^{2}  { y_1}\,  ^{2}}{
 \left( 16\,{x}^{2}+1 \right) x}}-\frac{1}{32}\,{\frac { \left( 4\,x+1
 \right) ^{3}  { y_2}\, ^{2}}{ \left(
16\,{x}^{2}+1 \right) x}}\\
-{\frac { \left( 2\,x-1 \right) { y_1}
 }{ \left( 16\,{x}^{2}+1 \right) x}}+\frac{1}{2}\,{\frac {
 \left( 8\,x-1 \right) { y_2}  }{ \left( 16\,{x}^{2}
+1 \right) x}}
\end{multline}
\begin{multline}\nonumber
g_2(x,\bfy)={\frac {1568}{625}}\,{\frac {4\,x-1}{ \left( 16\,{x}^{2}+1 \right) {x}
^{3}}}+\frac{1}{16}\,{\frac { \left( 4\,x+1 \right)  \left( 4\,x-1 \right) ^{2
}{ y_1}  { y_2}  }{ \left( 16\,{x}
^{2}+1 \right) x}}\\
+\frac{1}{32}\,{\frac { \left( 4\,x-1 \right) ^{3}  {
 y_1} \, ^{2}}{ \left( 16\,{x}^{2}+1 \right)
x}}+\frac{1}{32}\,{\frac { \left( 4\,x-1 \right)  \left( 4\,x+1 \right) ^{2}
  { y_2} \,^{2}}{ \left( 16\,{x}^{2}+1
 \right) x}}\\
 -\frac{1}{2}\,{\frac { \left( 8\,x+1 \right) { y_1}  }{ \left( 16\,{x}^{2}+1 \right) x}}+{\frac { \left( 2\,x+1
 \right) { y_2}  }{ \left( 16\,{x}^{2}+1 \right) x}}
\end{multline}

  \subsection{Calculation of the functions $F_n(\xi)$}\label{CalcHn}

Substituting the two scale expansion \eqref{eq:eq51} in \eqref{eq:eqp} we obtain an asymptotic series, for $1\ll x\ll \xi$ and $F_0(\xi)\ll x$, in integer powers of $x^{-1}$, with coefficients functions of $\xi$; the first term is
$$ {\xi}^{2}\, {\frac
{d^{2}}{d{\xi}^{2}}}F_{{0}} \left( \xi \right)  +  \xi\,
  {\frac {d}{d\xi}}F_{{0}} \left( \xi \right)-\frac{1}{2}\,  F_{{0}}(\xi) ^{2} -F_{{0
}} \left( \xi \right) =O \left( {x}^{-1} \right) $$
and we look for $F_0$ analytic at $\xi=0$ and $F_0(0)=0,\, F_0'(0)=1$.

Substituting $F_0(\xi)=G_0(s)$, $s=\ln\xi$ we get $G_0''-\tfrac{1}{2}G_0^2-G_0=0$, an equation  having, as expected, Weierstrass elliptic functions as general solution, a one parameter family of rational solutions, as well as  two constant solutions: multiplying the equation by $2G_0'$ we obtain $G_0'^2=\tfrac{1}{3}G_0^3+G_0^2+Const.$ whose solution contains a term $s=\ln\xi$ unless $Const.=0$, in which case we obtain $F_0(\xi)= 12\xi/[c(1-\xi/c)^{2}]$ (degenerate elliptic) and $F_0'(0)=1$ implies the formula in \eqref{eq:f11}.

The coefficient of $x^{-1}$ gives the equation for $F_1(\xi)$:
$$\xi^2F_1''+\xi F_1'-(1+F_0)F_1=-\xi^2 F_0''$$
which shows that the only possible singularities for $F_1$ are at $\xi=0$ and $\xi=12$. Similarly, the differential equation for all $F_n$ are linear, with coefficients depending on $F_0,\ldots,F_{n-1}$, and by induction, the only possible singularities for $F_1$ are at $\xi=0$ and $\xi=12$.

To determine $F_1$ we need two constants; one is determined from the condition that $F_1$ be analytic at $0$ (thus the coefficient multiplying $\ln\xi$ must vanish), and the other constant is determined at the next step, when solving for $F_2$ (from the condition that  $F_2$ does not contain $\ln\xi$ terms). This pattern continues for all $F_n$, and is typical for generic equations.

An additional potential obstruction to $F_n$ rational occurs at $n=6$: $F_6$ also contains, in principle a  term $\ln(\xi-12)$ multiplied by a constant; this term  vanishes precisely when the coefficient of $x^{-4}$  in \eqref{eq:eqp} equals $-\tfrac{392}{625}$: any other value of this coefficient produces an equation with movable branch points, hence not having the Painlev\'e property! This is the special feature of integrability of P$_I$.

For practical calculation of $F_n$, for $n\Ge 3$ it is better to substitute $F_n(\xi)={\frac {\xi\, \left( \xi+12 \right) }{ \left( \xi-12 \right) ^{3}}}V_n(\xi)$; the functions $V_n(\xi)$ can be calculated recursively using only two successive integrations of rational functions.

\end{document}